\documentclass[11pt]{article}

\usepackage[T1]{fontenc}
\usepackage[utf8]{inputenc}
\usepackage[margin=1in]{geometry}
\usepackage{amsmath,amssymb,amsthm}
\usepackage{booktabs,tabularx}
\usepackage{xcolor}
\usepackage{tikz}
\usepackage{microtype}
\usepackage{hyperref}

\hypersetup{
  colorlinks=true,
  linkcolor=blue!45!black,
  citecolor=blue!45!black,
  urlcolor=blue!55!black,
  pdftitle={A 60-Vertex Lower Bound for Cubic Bipartite Counterexamples to the Erdos--Gyarfas Conjecture},
  pdfauthor={Julius Tranquilli}
}
\newtheorem{theorem}{Theorem}
\newtheorem{lemma}[theorem]{Lemma}
\newtheorem{proposition}[theorem]{Proposition}
\newtheorem{corollary}[theorem]{Corollary}
\theoremstyle{definition}
\newtheorem{definition}[theorem]{Definition}

\newcommand{\cB}{\mathcal{B}}
\newcommand{\cH}{\mathcal{H}}
\newcommand{\code}[1]{\texttt{#1}}

\tikzset{
  pointnode/.style={circle,draw=blue!55!black,fill=blue!7,
    minimum size=6.5mm,inner sep=1pt,font=\small},
  blocknode/.style={rectangle,rounded corners=1.5pt,
    draw=orange!70!black,fill=orange!10,
    minimum width=9mm,minimum height=6.5mm,inner sep=1pt,font=\small},
  cycleedge/.style={very thick,blue!60!black},
  newedge/.style={very thick,red!70!black},
  auxiliaryedge/.style={semithick,gray!65},
  compatnode/.style={circle,fill=blue!65!black,inner sep=1.5pt}
}

\title{A 60-Vertex Lower Bound for Cubic Bipartite\\
Counterexamples to the Erd\H{o}s--Gy\'arf\'as Conjecture}
\author{Julius Tranquilli\\
\small\href{mailto:jtranqs@gmail.com}{\texttt{jtranqs@gmail.com}}}
\date{2 August 2026\\[0.5ex]
\small\href{https://doi.org/10.5281/zenodo.21695513}
{\texttt{DOI: 10.5281/zenodo.21695513}}}

\begin{document}
\maketitle

\begin{abstract}
A certified exhaustive computation shows that every simple cubic bipartite
graph on at most \(58\) vertices contains a cycle of length \(4\), \(8\), or
\(16\). Consequently, any cubic bipartite counterexample to the
Erd\H{o}s--Gy\'arf\'as conjecture has at least \(60\) vertices, improving the
established published lower bound of \(30\).

The proof begins with a Moore-bound observation: below \(62\) vertices, a
cubic bipartite graph avoiding \(4\)- and \(8\)-cycles must contain a
\(6\)-cycle. Viewing the graph as the Levi graph of a linear symmetric
\(v_3\)-configuration turns this \(6\)-cycle into a Berge triangle. Up to
symmetry, only two rooted extensions are possible. A complete
restricted-growth search on at most \(29\) points closes both search trees.
The computation is checked by two separately implemented searches using
different \(C_{16}\) oracles and by a static witness certificate. Source code,
certificates, and reproduction instructions are archived with the paper.
\end{abstract}



\section{Introduction}

The Erd\H{o}s--Gy\'arf\'as conjecture asks whether every finite simple graph
of minimum degree at least three contains a simple cycle whose length is a
power of two~\cite{Erdos1997}.  In a simple graph the first relevant lengths
are \(4,8,16,32,\ldots\).  The conjecture remains open, although it is known
for several restricted graph classes; examples include \(3\)-connected cubic
planar graphs~\cite{HeckmanKrakovski2013}, \(P_{10}\)-free
graphs~\cite{HuShen2024}, and, with computer assistance, \(P_{13}\)-free
graphs~\cite{HegdeSandeepShashank2025}.  Recent work also gives strong
structural restrictions on a minimal counterexample~\cite{Carr2026}.

Numerically, the result raises the established published lower bound
applicable to the cubic-bipartite class from \(n\geq30\) to \(n\geq60\), and
raises the newest public computational bound from \(n\geq32\) to
\(n\geq60\).

This paper concerns the finite-order frontier inside the class of simple
cubic bipartite graphs.  Its main result is deliberately stated in a
stronger form than a bare verification of the conjecture.

\begin{theorem}[Finite cubic-bipartite frontier]\label{thm:main}
Every simple cubic bipartite graph \(G\) with
\[
  |V(G)|\leq 58
\]
contains a simple cycle of length \(4\), \(8\), or \(16\).
\end{theorem}

\begin{corollary}\label{cor:bound}
Any simple cubic bipartite counterexample to the Erd\H{o}s--Gy\'arf\'as
conjecture has at least \(60\) vertices.
\end{corollary}

\begin{proof}
Let \(L\) and \(R\) be the two bipartition classes of a counterexample \(G\).
Cubicity gives
\[
  3|L|=|E(G)|=3|R|,
\]
so \(|L|=|R|\) and \(|V(G)|\) is even.  Theorem~\ref{thm:main} excludes all
possible orders through \(58\); the next possible order is \(60\).
\end{proof}

\subsection{Comparison with previous bounds}

Markstr\"om's 2004 computation established the customary cubic lower bound
\(n\geq30\)~\cite{Markstrom2004}. The accessible abstract of a 2011
bipartite computation also states \(30\), although later sources attribute
\(32\) to that work
\cite{NowbandeganiEsfandiari2011,NowbandeganiEtAl2014,HuShen2024}. Separately,
a public 2026 SAT/SAT-Modulo-Symmetries computation excludes all
minimum-degree-three graphs through order \(31\), giving the more broadly
applicable bound \(n\geq32\)~\cite{Balaji2026}. Table~\ref{tab:comparison}
records both comparisons.

\begin{table}[ht]
\centering
\small
\caption{Comparison with the previously applicable finite-order bounds.}
\label{tab:comparison}
\begin{tabularx}{\textwidth}{@{}Xcccl@{}}
\toprule
Comparison source & Previous & Present & Increase & Newly excluded
cubic-bipartite orders\\
\midrule
Published cubic result (Markstr\"om, 2004)
  & \(n\geq30\) & \(n\geq60\) & \(30\)
  & \(30,32,\ldots,58\) (15 orders)\\
Public minimum-degree-three computation (2026)
  & \(n\geq32\) & \(n\geq60\) & \(28\)
  & \(32,34,\ldots,58\) (14 orders)\\
\bottomrule
\end{tabularx}
\end{table}

Theorem~\ref{thm:main} also shows that one of the first three relevant
power-of-two cycle lengths is forced; no \(32\)-cycle is needed.

\subsection{Proof outline}

The proof has two ingredients: a short structural reduction and a finite
certified search. A cubic bipartite graph is first represented as the Levi
graph of a symmetric \(3\)-uniform, \(3\)-regular set system. An edge-rooted
Moore bound then forces a \(C_6\), hence a Berge triangle. After normalizing
that triangle, only two root orbits remain. A universal restricted-growth
search with point cap \(29\) exhausts both orbits and recognizes a completed
configuration as soon as all introduced points are cubic. The resulting
certificate covers the entire range \(v\leq29\) at once.

\section{Incidence configurations and cycle translations}

Let \(G=(X,Y;E)\) be a connected simple cubic bipartite graph.  As in the
proof of Corollary~\ref{cor:bound}, write
\[
  |X|=|Y|=:v,\qquad |V(G)|=2v.
\]
For each \(y\in Y\), let \(B_y=N_G(y)\) and regard
\(\cB=(B_y)_{y\in Y}\) as an indexed family of three-element blocks on the
point set \(X\).  Repeated triples are allowed at this stage: distinct
vertices of \(Y\) remain distinct block indices even if they have the same
neighborhood.  The family has \(v\) indexed blocks, and each point belongs to
exactly three of them.

\begin{proposition}[Incidence translation]\label{prop:incidence}
The neighborhood construction is a bijection, up to the natural relabelings,
between connected simple cubic bipartite graphs with specified bipartition
\((X,Y)\) and connected \(3\)-uniform, \(3\)-regular incidence structures
\((X,(B_y)_{y\in Y})\) having \(v\) points and \(v\) indexed blocks.
\end{proposition}

\begin{proof}
The preceding construction gives the incidence structure.  Conversely, make
one graph vertex for each point and one for each block index \(y\), and join
\(x\) to \(y\) precisely when \(x\in B_y\).  Uniformity gives degree three on
the block side, regularity gives degree three on the point side, and the two
kinds of objects form a bipartition.  Each membership is a Boolean relation,
so there is at most one edge between a given point and block index even when
two indexed blocks are equal.  The two notions of connectedness are the same
because an incidence walk is exactly a graph walk in the constructed
bipartite graph.
\end{proof}

\begin{definition}
An indexed triple system is \emph{linear} if blocks with distinct indices
have at most one common point; equivalently, no unordered pair of points
occurs in two indexed blocks.
\end{definition}

\begin{lemma}[\(C_4\)]\label{lem:c4}
The incidence graph contains a simple \(4\)-cycle if and only if two distinct
blocks contain the same pair of points.
\end{lemma}

\begin{proof}
A \(4\)-cycle in a bipartite incidence graph alternates as
\[
  x,B,x',B',x
\]
with \(x\neq x'\) and \(B\neq B'\).  Thus \(x,x'\in B\cap B'\).  Conversely,
two blocks sharing distinct points \(x,x'\) give exactly this \(4\)-cycle.
\end{proof}

By Lemma~\ref{lem:c4}, it suffices after the first rejection test to work with
linear triple systems.  These are the symmetric combinatorial
\(v_3\)-configurations of the configuration literature; their Levi graphs are
exactly the cubic bipartite graphs of girth at least six
\cite{BettenBrinkmannPisanski2000,Boben2007}.

A Berge cycle of length \(k\geq3\) consists of distinct blocks
\(B_0,\ldots,B_{k-1}\) and distinct points \(p_0,\ldots,p_{k-1}\) with
\(p_i\in B_i\cap B_{i+1}\), where indices are read modulo \(k\).
Equivalently, it is a simple \(C_{2k}\) in the Levi graph; in particular, a
\(C_6\) is a Berge triangle.

\paragraph{Worked example.}
Take
\[
  B_1=\{a,b,c\},\qquad B_2=\{c,d,e\},\qquad
  B_3=\{e,f,a\}.
\]
Their successive intersection points are \(c,e,a\). In the Levi graph these
incidences form the alternating cycle shown in
Figure~\ref{fig:incidence-triangle}.

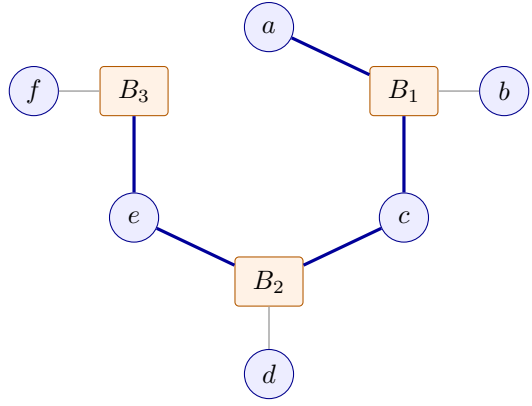
\begin{figure}[htbp]
\centering
\begin{tikzpicture}[scale=1.02]
  \node[pointnode] (a) at (0,1.65) {$a$};
  \node[blocknode] (B1) at (1.75,0.82) {$B_1$};
  \node[pointnode] (c) at (1.75,-0.82) {$c$};
  \node[blocknode] (B2) at (0,-1.65) {$B_2$};
  \node[pointnode] (e) at (-1.75,-0.82) {$e$};
  \node[blocknode] (B3) at (-1.75,0.82) {$B_3$};

  \node[pointnode] (b) at (3.05,0.82) {$b$};
  \node[pointnode] (d) at (0,-2.85) {$d$};
  \node[pointnode] (f) at (-3.05,0.82) {$f$};

  \draw[cycleedge] (a)--(B1)--(c)--(B2)--(e)--(B3)--cycle;
  \draw[auxiliaryedge] (B1)--(b);
  \draw[auxiliaryedge] (B2)--(d);
  \draw[auxiliaryedge] (B3)--(f);
\end{tikzpicture}
\caption{A Berge triangle and its Levi-graph \(C_6\). The blue cycle is
\(a-B_1-c-B_2-e-B_3-a\); the gray edges show the third point of each block.}
\label{fig:incidence-triangle}
\end{figure}

A familiar global example is the Fano plane: its seven points and seven lines
form a symmetric \(7_3\)-configuration, and its Levi graph is the cubic
bipartite Heawood graph.

\begin{lemma}[\(C_8\)]\label{lem:c8}
In a linear triple system, the incidence graph contains a simple \(8\)-cycle
if and only if there are four distinct blocks
\(B_0,B_1,B_2,B_3\) and four distinct points
\(p_0,p_1,p_2,p_3\) such that
\[
  p_i\in B_i\cap B_{i+1}\qquad (i\ \mathrm{mod}\ 4).
\]
In other words, the blocks contain a Berge quadrilateral.
\end{lemma}

\begin{proof}
A simple \(8\)-cycle alternates between four distinct point vertices and four
distinct block vertices.  Reading it cyclically gives the displayed
incidences.  Conversely, those incidences form the alternating closed walk
\[
 p_0,B_1,p_1,B_2,p_2,B_3,p_3,B_0,p_0.
\]
The stipulated distinctness makes this walk a simple \(8\)-cycle.  Linearity
ensures that a pair of consecutive blocks has at most one intersection
point, so the test is unambiguous.
\end{proof}

\begin{lemma}[Incremental \(C_{16}\) oracle]\label{lem:c16}
Let \(H\) be a partial incidence graph with no \(C_{16}\), and add a new block
vertex \(b\) adjacent to the three points in \(B=\{x,y,z\}\).  A new simple
\(C_{16}\) is created if and only if \(H\) contains a simple path of length
\(14\) between two distinct members of \(B\).
\end{lemma}

\begin{figure}[htbp]
\centering
\begin{tikzpicture}
  \node[pointnode] (x) at (-3,0) {$x$};
  \node[pointnode] (y) at (3,0) {$y$};
  \node[blocknode] (b) at (0,-1.35) {$b$};
  \node[pointnode] (z) at (0,-2.65) {$z$};

  \draw[cycleedge] (x) to[bend left=38]
    node[midway,above=4pt,fill=white,inner sep=2pt,font=\small]
    {old simple path in \(H\) (\(14\) edges)} (y);
  \draw[newedge] (x)--(b)--(y);
  \draw[auxiliaryedge] (b)--(z);
  \node[font=\small,text=red!70!black] at (0,-0.40)
    {two new cycle edges};
\end{tikzpicture}
\caption{The incremental \(C_{16}\) test. The old \(14\)-edge path and the
two highlighted edges through the new block vertex \(b\) form a
\(16\)-cycle. The third incidence \(bz\) is not used by this cycle.}
\label{fig:c16-oracle}
\end{figure}
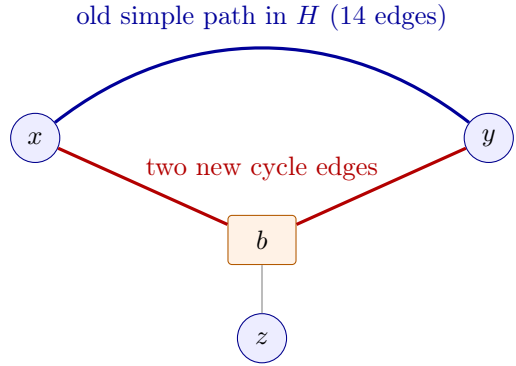

\begin{proof}
Any newly created \(C_{16}\) must use \(b\) and exactly two of its incident
edges. Deleting \(b\) from that cycle leaves a simple \(14\)-edge path in
\(H\). Conversely, adding \(b\) and the two corresponding incidence edges
closes any such path into a simple \(16\)-cycle, as in
Figure~\ref{fig:c16-oracle}.
\end{proof}

The statement allows one or two members of \(B\) to be newly introduced
points.  Such a point has degree zero in \(H\), so it cannot be an endpoint of
an old nontrivial path; the equivalence remains valid without a special case.

\section{Triangle-rooted proof}\label{sec:triangle-rooted}

\subsection{Moore reduction and two normalized roots}

\begin{lemma}[Edge-rooted Moore reduction]\label{lem:moore}
Every simple cubic bipartite graph on at most \(58\) vertices with no
\(C_4\) and no \(C_8\) contains a \(C_6\).
\end{lemma}

\begin{proof}
Suppose instead that \(G\) also has no \(C_6\), and choose an edge \(uv\) in
one of its components. Since \(G\) is simple and bipartite, that component
has girth at least \(10\). Starting from \(u\), without using \(uv\), expose
the nonbacktracking cubic tree through depth four; its level sizes are
\[
  1,2,4,8,16.
\]
Do the same from \(v\). A repeated vertex within one exposure gives a cycle
of length at most \(8\). An intersection between the two exposures gives,
together with \(uv\), a cycle of length at most \(9\), which is even by
bipartiteness and hence has length at most \(8\). Thus all exposed vertices
are distinct, so the component has at least
\[
  2(1+2+4+8+16)=62
\]
vertices, a contradiction.
\end{proof}

\begin{lemma}[Triangle-root orbits]\label{lem:triangle-orbits}
Let \(\cH\) be a linear \(3\)-uniform configuration containing a Berge
triangle. After relabeling, its three triangle blocks are
\[
  \{0,1,3\},\qquad \{1,2,4\},\qquad \{0,2,5\}.
\]
Up to the stabilizer of this rooted triangle, the final block through point
\(0\) is one of
\[
  \{0,4,6\},\qquad \{0,6,7\}.
\]
\end{lemma}

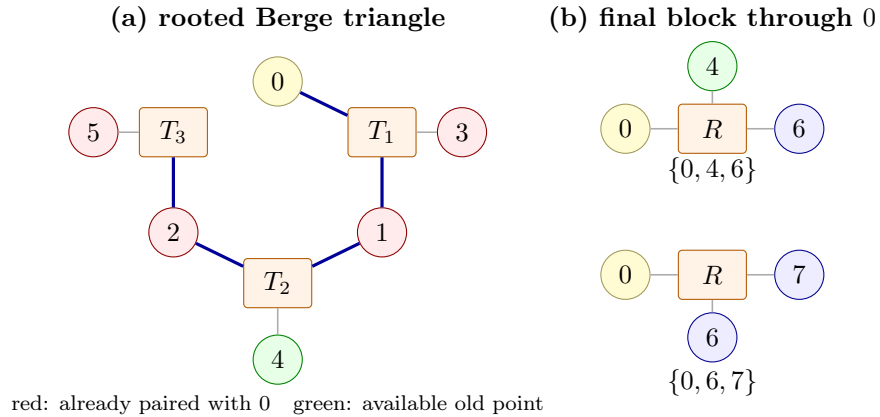
\begin{figure}[htbp]
\centering
\begin{tikzpicture}[scale=0.92]
  \begin{scope}[shift={(-2.9,0)}]
    \node[font=\small\bfseries] at (0,2.35) {(a) rooted Berge triangle};
    \node[pointnode,fill=yellow!22,draw=yellow!55!black] (p0) at (0,1.45) {$0$};
    \node[blocknode] (T1) at (1.5,0.72) {$T_1$};
    \node[pointnode,fill=red!8,draw=red!55!black] (p1) at (1.5,-0.72) {$1$};
    \node[blocknode] (T2) at (0,-1.45) {$T_2$};
    \node[pointnode,fill=red!8,draw=red!55!black] (p2) at (-1.5,-0.72) {$2$};
    \node[blocknode] (T3) at (-1.5,0.72) {$T_3$};
    \node[pointnode,fill=red!8,draw=red!55!black] (p3) at (2.65,0.72) {$3$};
    \node[pointnode,fill=green!10,draw=green!50!black] (p4) at (0,-2.55) {$4$};
    \node[pointnode,fill=red!8,draw=red!55!black] (p5) at (-2.65,0.72) {$5$};

    \draw[cycleedge] (p0)--(T1)--(p1)--(T2)--(p2)--(T3)--cycle;
    \draw[auxiliaryedge] (T1)--(p3);
    \draw[auxiliaryedge] (T2)--(p4);
    \draw[auxiliaryedge] (T3)--(p5);
    \node[align=center,font=\scriptsize] at (0,-3.2)
      {red: already paired with \(0\)\quad green: available old point};
  \end{scope}

  \begin{scope}[shift={(3.35,0)}]
    \node[font=\small\bfseries] at (0,2.35)
      {(b) final block through \(0\)};

    \begin{scope}[shift={(0,-0.18)}]
    \node[blocknode] (R1) at (0,0.95) {$R$};
    \node[pointnode,fill=yellow!22,draw=yellow!55!black] (q0) at (-1.25,0.95) {$0$};
    \node[pointnode,fill=green!10,draw=green!50!black] (q4) at (0,1.85) {$4$};
    \node[pointnode] (q6) at (1.25,0.95) {$6$};
    \draw[auxiliaryedge] (q0)--(R1)--(q4);
    \draw[auxiliaryedge] (R1)--(q6);
    \node[font=\small] at (0,0.35) {$\{0,4,6\}$};

    \node[blocknode] (R2) at (0,-1.15) {$R$};
    \node[pointnode,fill=yellow!22,draw=yellow!55!black] (r0) at (-1.25,-1.15) {$0$};
    \node[pointnode] (r6) at (0,-2.05) {$6$};
    \node[pointnode] (r7) at (1.25,-1.15) {$7$};
    \draw[auxiliaryedge] (r0)--(R2)--(r6);
    \draw[auxiliaryedge] (R2)--(r7);
    \node[font=\small] at (0,-2.7) {$\{0,6,7\}$};
    \end{scope}
  \end{scope}
\end{tikzpicture}
\caption{The forced triangle and its two normalized extensions. Here
\(T_1=\{0,1,3\}\), \(T_2=\{1,2,4\}\), and
\(T_3=\{0,2,5\}\). Linearity excludes \(1,2,3,5\) from the final block
through \(0\), leaving either the old point \(4\) and one new point, or two
new points.}
\label{fig:triangle-roots}
\end{figure}

\begin{proof}
The three intersection points and the three remaining triangle points are
distinct, so they may be labeled as in Figure~\ref{fig:triangle-roots}.
The two existing blocks through \(0\) pair it with \(1,2,3,5\), which
linearity excludes from its final block. The only available old point is
\(4\). Thus the block contains either \(4\) and one new point or two new
points. First-occurrence labeling gives \(6\), or \(6,7\), and the
rooted-triangle stabilizer makes all choices within each form equivalent.
\end{proof}

\paragraph{Restricted-growth extension.}
After installing one of the two roots, the search labels points in order of
first occurrence. At each state it chooses the least introduced point \(p\)
of degree below three and proposes its remaining blocks
\(\{p,q,r\}\) in lexicographic order. Eligible old labels exceed \(p\), have
degree below three, and have not already been paired with \(p\); a proposal
may instead use the next fresh label, or the next two fresh labels together.
The search rejects a proposal that violates a degree or pair constraint or
creates a Berge cycle of length \(4\) or \(8\), and otherwise inserts it and
recurses. It records a completion as soon as every introduced point has
degree three, whether or not the cap of \(29\) points has been reached.

\begin{proposition}[Triangle-rooted coverage]\label{prop:triangle-coverage}
Every connected linear symmetric \(v_3\)-configuration with \(v\leq29\)
that contains a Berge triangle and has no Berge cycles of lengths \(4\) or
\(8\) occurs as a completed state in one of the two triangle-rooted
cap-\(29\) restricted-growth trees.
\end{proposition}

\begin{proof}
Choose a Berge triangle and apply Lemma~\ref{lem:triangle-orbits}. Install
its three blocks and the appropriate fourth block through point \(0\).
The preinstalled block \(\{1,2,4\}\) is lexicographically first among the
blocks with least point \(1\): any other such block cannot reuse a point
already paired with \(1\). Begin the ordinary restricted-growth recursion at
point \(1\), after this block.

Maintain the following invariant before processing a point \(p\): every
target block with smaller least point has been inserted; the inserted target
blocks through \(p\) form a lexicographic initial segment; and introduced
labels are consecutive and follow first occurrence. The normalized root
establishes the invariant at \(p=1\).

Let \(B\) be the first target block through \(p\) not yet inserted. Every
block containing \(p\) and a smaller point was inserted when that smaller
point was processed, so the other two points of \(B\) have labels greater
than \(p\). Previously unseen points receive the next one or two labels.
Thus \(B\) occurs among the generator's proposals and is later than the
preceding block through \(p\). It passes the degree and pair tests because
the target is linear, and it passes the cycle tests because the target has no
Berge cycles of lengths \(4\) or \(8\). Inserting \(B\) preserves the
invariant; when \(p\) becomes cubic, the recursion advances to the least
unfinished introduced point.

If the introduced points formed a proper subset closed under their incident
blocks, the incidence graph would be disconnected. Hence connectedness
ensures that every target point is eventually introduced. Since there are at
most \(29\) points, no label outside the cap is needed. Once all introduced
points are cubic, the tree recognizes the completed configuration
immediately, even if fewer than \(29\) points were introduced.
\end{proof}

\subsection{Certified finite search}

One implementation uses simple-path DFS for its \(C_{16}\) oracle; a second
joins complete lists of \(7\)-edge half-paths. Both use \(29\) as a point cap
and test for a completion whenever all introduced points have degree three.
Their cap-\(29\) totals are:

\begin{table}[ht]
\centering
\small
\caption{Universal cap-$29$ triangle-rooted search, split by root orbit.}
\label{tab:triangle-orbits}
\begin{tabular}{rrrrrrr}
\toprule
orbit & states & attempted & structural & $C_8$ & $C_{16}$ & completions\\
\midrule
1 & 1,405 & 106,964 & 7,184 & 63,526 & 34,850 & 0\\
2 & 20,088 & 1,655,404 & 113,012 & 1,208,473 & 313,832 & 0\\
\midrule
Total & 21,493 & 1,762,368 & 120,196 & 1,271,999 & 348,682 & 0\\
\bottomrule
\end{tabular}
\end{table}

A static certificate contains one stream for each triangle-root orbit. The
streaming checker reconstructs the appropriate root, every state, and every
candidate. It recomputes structural
rejections, validates positive \(C_8\) and \(C_{16}\) witnesses, recursively
checks expansion records, and rejects any completed configuration. Crucially,
completion is tested at the number of points actually introduced, rather
than only at the cap. Thus each stream simultaneously covers every smaller
side size represented by its root orbit.

\begin{proposition}[Certified universal triangle search]
\label{prop:triangle-search}
The two searches were implemented separately and use different \(C_{16}\)
oracles; they agree on every counter and transcript hash for both cap-\(29\)
roots. A third streaming checker accepts both certificate streams with zero
completions. Consequently, no connected linear symmetric
\(v_3\)-configuration with \(v\leq29\) that
contains a Berge triangle avoids Berge cycles of lengths \(4\) and \(8\).
\end{proposition}

\begin{proof}
The checker follows the deterministic candidate schedule: each
rejection has a condition or positive cycle witness checked directly, while
every expansion recursively consumes its child stream. An induction
over this recursion shows that an accepted stream accounts for every proposal
in its rooted search tree. Malformed, truncated, trailing, counter-tampered,
or witness-tampered streams are rejected. Accepted streams therefore close
both cap-\(29\) trees. Proposition~\ref{prop:triangle-coverage} converts this
tree exhaustion into the stated finite result.
\end{proof}

\begin{proof}[Computer-assisted proof of Theorem~\ref{thm:main}]
Suppose \(G\) is a simple cubic bipartite graph on at most \(58\) vertices
with no \(C_4,C_8,\) or \(C_{16}\), and choose a connected component \(H\).
Every component of a cubic graph is cubic. Lemma~\ref{lem:moore} gives a
\(C_6\) in \(H\). Proposition~\ref{prop:incidence} translates \(H\) into a
connected symmetric \(v_3\)-configuration with
\(v=|V(H)|/2\leq29\). Lemma~\ref{lem:c4} makes it linear, and the \(C_6\)
becomes a Berge triangle. The absent \(C_8\) and \(C_{16}\) become absent
Berge cycles of lengths \(4\) and \(8\), contradicting
Proposition~\ref{prop:triangle-search}.
\end{proof}

\subsection{Six deepest kernels}

The state-dumping checker reconstructs \(337\) surviving states with
\(19\) blocks, the maximum depth attained by the triangle-rooted search.
Every such state has already introduced all \(29\) points allowed by the cap.
The retained mapping certificate supplies a point permutation and
block permutation from every labeled state to one of six representatives.

\begin{table}[ht]
\centering
\small
\caption{The six color-preserving classes among the deepest triangle-rooted states.}
\label{tab:triangle-kernels}
\begin{tabular}{rrrr}
\toprule
kernel & labelled occurrences & deficient points & compatible-pair graph\\
\midrule
$K_1$ & 2 & 17 & $3K_{1,2}$\\
$K_2$ & 20 & 18 & $K_2\sqcup2K_{1,2}$\\
$K_3$ & 20 & 18 & $K_2\sqcup2K_{1,2}$\\
$K_4$ & 75 & 18 & $2K_{1,2}$\\
$K_5$ & 200 & 19 & $K_{1,2}$\\
$K_6$ & 20 & 19 & $K_{1,2}$\\
\midrule
Total & 337 & &\\
\bottomrule
\end{tabular}
\end{table}

In the table and in Figure~\ref{fig:compatibility-forests}, isolated deficient
points are omitted from the displayed graph types.

For one of these states, call a point \emph{deficient} when its degree is
less than three. Form a graph on the deficient points by joining \(x\) and
\(y\) precisely when they do not already share a block and the old incidence
graph contains neither a simple \(6\)-edge path nor a simple \(14\)-edge path
between them. Call such a pair compatible.

A new block \(\{x,y,z\}\) can avoid \(C_4,C_8,\) and \(C_{16}\) only if all
three pairs are compatible: an old shared block gives a \(C_4\), while an
old path of length \(6\) or \(14\) is closed by the new block into a \(C_8\)
or \(C_{16}\), respectively. A legal new block therefore requires a triangle
in the compatible-pair graph.

\begin{figure}[htbp]
\centering
\begin{tikzpicture}[scale=0.95]
  \newcommand{\compatstar}[2]{%
    \begin{scope}[shift={(#1,#2)}]
      \draw[semithick,blue!60!black] (0,0.24)--(-0.32,-0.24)
        (0,0.24)--(0.32,-0.24);
      \node[compatnode] at (0,0.24) {};
      \node[compatnode] at (-0.32,-0.24) {};
      \node[compatnode] at (0.32,-0.24) {};
    \end{scope}}
  \newcommand{\compatedge}[2]{%
    \begin{scope}[shift={(#1,#2)}]
      \draw[semithick,blue!60!black] (-0.3,0)--(0.3,0);
      \node[compatnode] at (-0.3,0) {};
      \node[compatnode] at (0.3,0) {};
    \end{scope}}

  \begin{scope}[shift={(-3.8,1.15)}]
    \draw[gray!35,rounded corners=2pt] (-1.6,-0.65) rectangle (1.6,1.1);
    \node[font=\small] at (0,0.82) {$K_1:\ 3K_{1,2}$};
    \compatstar{-1}{0}
    \compatstar{0}{0}
    \compatstar{1}{0}
  \end{scope}
  \begin{scope}[shift={(0,1.15)}]
    \draw[gray!35,rounded corners=2pt] (-1.6,-0.65) rectangle (1.6,1.1);
    \node[font=\small] at (0,0.82) {$K_2:\ K_2\sqcup2K_{1,2}$};
    \compatedge{-1.05}{0}
    \compatstar{0}{0}
    \compatstar{1}{0}
  \end{scope}
  \begin{scope}[shift={(3.8,1.15)}]
    \draw[gray!35,rounded corners=2pt] (-1.6,-0.65) rectangle (1.6,1.1);
    \node[font=\small] at (0,0.82) {$K_3:\ K_2\sqcup2K_{1,2}$};
    \compatedge{-1.05}{0}
    \compatstar{0}{0}
    \compatstar{1}{0}
  \end{scope}

  \begin{scope}[shift={(-3.8,-1.15)}]
    \draw[gray!35,rounded corners=2pt] (-1.6,-0.65) rectangle (1.6,1.1);
    \node[font=\small] at (0,0.82) {$K_4:\ 2K_{1,2}$};
    \compatstar{-0.55}{0}
    \compatstar{0.55}{0}
  \end{scope}
  \begin{scope}[shift={(0,-1.15)}]
    \draw[gray!35,rounded corners=2pt] (-1.6,-0.65) rectangle (1.6,1.1);
    \node[font=\small] at (0,0.82) {$K_5:\ K_{1,2}$};
    \compatstar{0}{0}
  \end{scope}
  \begin{scope}[shift={(3.8,-1.15)}]
    \draw[gray!35,rounded corners=2pt] (-1.6,-0.65) rectangle (1.6,1.1);
    \node[font=\small] at (0,0.82) {$K_6:\ K_{1,2}$};
    \compatstar{0}{0}
  \end{scope}
\end{tikzpicture}
\caption{The compatibility graphs of the six terminal kernels, with isolated
vertices omitted. Every displayed graph is a forest and therefore contains
no triangle.}
\label{fig:compatibility-forests}
\end{figure}
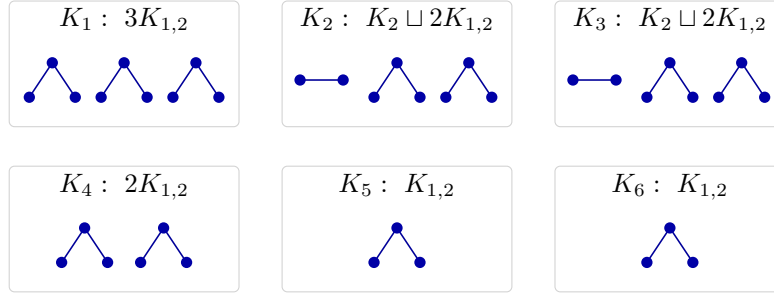

\begin{proposition}[Six deepest kernels]\label{prop:triangle-kernels}
The \(337\) triangle-rooted states with \(19\) blocks form the six
color-preserving isomorphism classes in
Table~\ref{tab:triangle-kernels}. On their existing \(29\)-point sets, none
admits a twentieth block without creating a \(C_4,C_8,\) or \(C_{16}\) in
its Levi graph.
\end{proposition}

\begin{proof}
The mapping certificate is checked row by row against the \(337\)
states reconstructed from the two witness streams. The independent Python
checker verifies that every point map and block map is a permutation and
preserves every incidence. It then recomputes simple paths of lengths
\(6\) and \(14\) for every pair of deficient points in each representative.
This gives the forests in Figure~\ref{fig:compatibility-forests}; restoring
the omitted isolated vertices preserves triangle-freeness. Hence no
representative admits another block within the \(29\)-point cap.
\end{proof}

The six forests explain the deepest terminal obstruction; branches terminating
before \(19\) blocks remain covered directly by the two universal witness
streams.

\section{Additional checks}\label{sec:crosschecks}

Agreement between the two triangle-rooted searches holds at the level of
ordered decision-transcript hashes, not only aggregate counters. An earlier
arbitrary-root enumeration, which does not assume a Berge triangle, also
gives zero completions throughout \(v\leq29\). The repository contains the
full outputs and verification tests for both computations.

For a generator-level comparison, the \(C_8\) and \(C_{16}\) filters were
disabled and all completed connected linear symmetric
\(v_3\)-configurations were generated for \(v=7,\ldots,13\). After
color-preserving canonical labeling and deduplication, the resulting graph6
sets agree exactly with nauty 2.9.3 \code{genbg} output generated with degree
three and at most one common neighbor on the block side
\cite{McKayPiperno2014}.

\begin{table}[ht]
\centering
\small
\caption{Set-level generator comparison. The second column counts rooted
and labelled restricted-growth leaves before deduplication; the third is the
common number of color-preserving canonical graph6 records.}
\label{tab:genbg}
\begin{tabular}{rrr}
\toprule
\(v\) & Rooted/labelled leaves & Canonical configurations\\
\midrule
7  & 1         & 1\\
8  & 4         & 1\\
9  & 44        & 3\\
10 & 496       & 10\\
11 & 7,840     & 31\\
12 & 136,575   & 229\\
13 & 2,337,152 & 2,036\\
\bottomrule
\end{tabular}
\end{table}

The common canonical counts in Table~\ref{tab:genbg} also agree with the
published census of Betten, Brinkmann, and Pisanski
\cite{BettenBrinkmannPisanski2000}. This independent overlap checks the
restricted-growth generator through \(v=13\); completeness beyond that range
follows from Proposition~\ref{prop:triangle-coverage}.

As positive controls, two independent cycle enumerators were applied to
\(128\) symmetric \(19_3\)-configurations known to avoid \(C_4\) and \(C_8\)
while containing \(C_{16}\); both classified every input correctly. The
primary implementations and certificate checker still share the mathematical
coverage argument in Proposition~\ref{prop:triangle-coverage}. The computation
does not examine side size \(30\), and hence does not exclude an order-\(60\)
counterexample.

\section{Conclusion}

The Moore reduction and incidence translation leave two triangle-rooted
cap-\(29\) searches. The certified search rules out a completion in either
orbit, and its \(337\) deepest states collapse to six terminal kernels with
triangle-free compatibility graphs. Together, the structural reduction and
certified enumeration prove that every simple cubic bipartite graph on at most
\(58\) vertices contains a \(4\)-, \(8\)-, or \(16\)-cycle. The lower bound
for any cubic bipartite counterexample is therefore \(60\).

\section*{Artifact availability}

The preprint is archived on Zenodo at
\href{https://doi.org/10.5281/zenodo.21695513}
{\code{DOI: 10.5281/zenodo.21695513}} as version \code{v1.0.0}. Source code,
certificates, verification programs, logs, and reproduction instructions are
available from the accompanying
\href{https://github.com/floor-licker/erdos-gyarfas-cubic-bipartite}
{GitHub repository}, which also provides an immutable release, a complete
SHA-256 manifest, and research-provenance documentation.

\section*{AI disclosure}

OpenAI ChatGPT materially assisted the initial research, including developing
computational and structural approaches, portions of the incidence-based code
and preliminary arguments, running and interpreting computations, and an
initial literature audit. OpenAI Codex subsequently assisted with code and
artifact auditing, reproducibility checks, additional verifiers and
certificates, integration of the triangle-rooted method, and drafting and
editing the manuscript and documentation. A custom closed-source coding and
formalization harness built on a fork of Codex also assisted with later code,
formalization, and verification work; it is not included in the public
artifact or treated as independent evidence. Julius Tranquilli selected and
directed the project and accepts responsibility for the final work; the
retained code, logs, exact outputs, certificates, and mathematical arguments,
rather than AI assertions, form the evidentiary basis, and a fuller
activity-level account is included in the repository.

\end{document}